\documentclass[12pt, letterpaper]{amsart}

\usepackage{amsmath, amssymb}
\usepackage{array}
\usepackage[frame,cmtip,arrow,matrix,line,graph,curve]{xy}
\usepackage{graphpap, color, paralist, pstricks}
\usepackage[mathscr]{eucal}
\usepackage[pdftex]{graphicx}
\usepackage[pdftex,colorlinks,backref=page,citecolor=blue]{hyperref}
\usepackage{tikz}
\usepackage{tikz-cd}
\usepackage{mathtools}
\usepackage{enumitem}

\makeatletter
\@namedef{subjclassname@2020}{%
  \textup{2020} Mathematics Subject Classification}
\makeatother

\newtheorem{theorem}{Theorem}[section]

\newtheorem{lemma}[theorem]{Lemma}

\theoremstyle{definition}
\newtheorem{definition}[theorem]{Definition}
\newtheorem{example}[theorem]{Example}

\newtheorem{remark}[theorem]{Remark}

\renewcommand{\AA}{\mathbb{A} }

\newcommand{\FF}{\mathbb{F} }

\newcommand{\PP}{\mathbb{P} }

\newcommand{\ZZ}{\mathbb{Z} }

\newcommand{\cF}{\mathcal{F} }
\newcommand{\cG}{\mathcal{G} }

\newcommand{\cO}{\mathcal{O} }

\def\GL{\mathrm{GL}}
\def\SL{\mathrm{SL}}

\def\PGL{\mathrm{PGL}}

\title{Minimal algebraic space filling curves on the product of projective lines}

\author{Menhaz Ahammed}
\address{Department of Mathematics, Fordham University, New York, NY 10023}
\email{mahammed@fordham.edu}

\author{Matthew Campbell}
\address{Department of Mathematics, University of Maryland, College Park, MD 20742}
\email{mcam12@umd.edu}

\author{Han-Bom Moon}
\address{Department of Mathematics, Fordham University, New York, NY 10023}
\email{hmoon8@fordham.edu}

\date{\today}

\begin{document}

\maketitle

\begin{abstract}
We investigate minimal degree smooth algebraic space filling curves on the product of projective lines. We prove that there are plenty of examples in an explicit sense, extending the existence result of Homma and Kim. 
\end{abstract}

\section{Introduction}\label{sec:intro}

Many geometric notions, such as smoothness and connectedness, can be extended to algebraic geometry over arbitrary base fields. Often, over a finite field, we obtain many counterintuitive results. In this paper, we explore one example of such results on extremal subvarieties. 

The original \textbf{space filling curve}, an example of a continuous surjective map $f : [0, 1] \to [0, 1]^2$ obtained by Peano \cite{Pea90}, has been very influential in the development of analysis and topology. Later, it was generalized to the Hahn-Mazurkiewicz theorem -- for any connected compact manifold $M$, there is a surjective continuous map $f : [0, 1] \to M$. However, we may not expect good geometric properties for a space filling curve. For example, it is not differentiable and not injective. 

One may ask a similar question to a variety over a finite field. Let $p$ be a prime and $q$ be a power of $p$. Let $\FF_q$ be the finite field of order $q$. For any smooth projective variety $X$ over $\FF_q$, let $X(\FF_q)$ be the set of $\FF_q$-rational points. The main question is the following: Can we find a morphism $f : C \to X$ from an irreducible curve $C$ such that $f(C)(\FF_q) = X(\FF_q)$? In other words, is there an algebraic space filling curve? The answer is, surprisingly, that one can find such a curve $C$ that is smooth, and $f$ is indeed an embedding. One may even find such a subcurve $C$ that is a complete intersection in $X$. We call such a curve a \textbf{smooth algebraic space filling curve}.

In \cite{Kat99}, Katz constructed an algebraic smooth space filling curve of $\AA^n$. His construction did not extend to projective spaces. Thus, he asked if there is a smooth space filling curve for a smooth projective variety. This question was answered affirmatively by Gabber and Poonen \cite{Gab01, Poo04}, independently.

Although we have a satisfactory theoretical answer for Katz's question, the proof is not constructive, so we do not have many explicit examples. For $\PP^2$ over $\FF_q$, the minimal degree of a smooth space filling curve is $q+2$, and these curves were classified by Homma and Kim \cite{HK13}. Higher degree space filling curves of $\PP^2$ were constructed in \cite{AG23}. Over $\PP^3$, some concrete examples over small finite fields were constructed in \cite{CD+23} using computer-aided calculation. 

In \cite{HK23}, Homma and Kim investigated smooth space filling curves in $(\PP^1 \times \PP^1)$ over $\FF_q$ with minimal degree. They showed that the lower bound of the bidegree of such a curve is $(q+1, q+1)$, and in this case, the defining equation of the curve is of the form 
\begin{equation}\label{eqn:mspcgenera}
    f(Y_0, Y_1)(X_0^q X_1 - X_0 X_1^q) + 
    g(X_0, X_1)(Y_0^q Y_1 - Y_0 Y_1^q)
\end{equation}
for two homogeneous polynomials $f \in \FF_q[Y_0,Y_1]_{q+1}$ and $g \in \FF_q[X_0, X_1]_{q+1}$. We will denote this curve by $C_{f,g}$. They also explicitly constructed examples of $f, g$ providing a smooth space filling curve over $q \ne 2$. 

The main result of this paper strengthens their existence statement. We show that there are \emph{many} smooth space filling curves of this form, in the following explicit sense. 

\begin{theorem}\label{thm:mainthm}
Let $q$ be an odd prime power. Let $f \in \FF_q[Y_0, Y_1]_{q+1}$ be a homogeneous polynomial such that $V(f) \subset \PP^1$ has no $\FF_{q}$-rational point. Then there is $g \in \FF_q[X_0, X_1]_{q+1}$ such that $C_{f,g}$ is a smooth space filling curve.
\end{theorem}

Moreover, as we will provide some statistical data in Section \ref{sec:smallprime}, most pairs $(f, g)$ without $\FF_q$-rational points yield a smooth curve $C_{f,g}$. 

This result is quite surprising, particularly compared to the result in \cite{CD+23}. In \cite{CD+23}, numerical results shows that smooth space filling curves in $\PP^3$ are quite rare. On the other hand, Theorem \ref{thm:mainthm} with the statistics in Section \ref{sec:smallprime} suggest that the only essential obstruction to being a smooth space filling curve is having $\FF_q$-rational points when we choose $f$ and $g$. 

We also show that we can construct smooth space filling curves with extra symmetry. We say a space filling curve is \textbf{symmetric} if it is of the form $C_{f,f}$, hence admits an extra $\ZZ/2\ZZ$-symmetry exchanging two factors. 

\begin{theorem}\label{thm:symmetric}
Let $q$ be an odd prime power. There exists a smooth symmetric space filling curve $C_{f,f}$ over $\FF_q$. 
\end{theorem}

The paper is organized as follows. In Section \ref{sec:preliminary}, we review some known results on space filling curves on $\PP^1 \times \PP^1$. Section \ref{sec:smoothness} shows a smoothness criterion, which is a small extension of a key lemma in \cite{HK23}. In Section \ref{sec:smallprime}, we present computational results for small primes, providing an evidence for Theorem \ref{thm:mainthm}. Sections \ref{sec:construction} and \ref{sec:symmetric} prove Theorems \ref{thm:mainthm} and \ref{thm:symmetric}, respectively. 

\subsection*{Notation and convention}

For a field $k$, the ring of polynomials with $k$-coefficients is denoted by $k[X_0, X_1, \dots, X_n]$. The $k$-subspace of homogeneous polynomials of (total) degree $d$ is denoted by $k[X_0, X_1, \dots, X_n]_d$. When we want to describe the $k$-space of bihomogeneous polynomials, we use the notation $k[X_0, \dots, X_m; Y_0, \dots, Y_n]_{d,e}$. We use capital letters $X_i, Y_j$ to denote variables, while small letters $x, y$ to denote explicit coordinates in some field $k$. 

\subsection*{Acknowledgements}

We thank Joshua Morales, Nicholas Padilla, and Ziyu Tian for helpful conversations.


\section{Preliminary}\label{sec:preliminary}

\subsection{Review of schemes}

For a field $k$, as a set, $\PP^n$ is the set of homogeneous coordinates $(x_0:x_1:\dots : x_n)$ such that $x_i \in k$. A \textbf{projective $k$-scheme} $X$ is a common zero set $V(F_1, F_2, \dots, F_k) \subset \PP^n$, where $F_i \in k[X_0, X_1, \dots, X_n]$ is a homogeneous polynomial. If we want to emphasize its field of definition, we will use the notation $\PP^n_k$ and $X_k$. 

For any extension field $K$ of $k$, we may consider the set of homogeneous coordinates $P = (x_0:x_1:\dots : x_n)$ with $x_i \in K$, such that $F_i(P) = 0$ for all defining equations of $X$. Then $P$ is called a \textbf{$K$-rational point}, and the set of $K$-rational points of $X$ is denoted by $X(K)$. 

These definitions can be naturally extended to a subscheme in a multi-projective space, that is, a product of projective spaces. In this case, to obtain a well-defined zero-set $X = V(F_1, F_2, \dots, F_k)$, we need multi-homogeneous polynomials $F_i$. The set of $K$-rational points are naturally defined. 

We are particularly interested in the case of a subscheme in the product of projective lines $\PP^1 \times \PP^1$ over a finite field $\FF_q$. In this case, the set of $\FF_q$-rational points $(\PP^1 \times \PP^1)(\FF_q)$ has $(q+1)^2$ points. We explore a curve $C$ in $(\PP^1 \times \PP^1)_{\FF_q}$. Since it is of codimension one, it is defined by a single bihomogeneous polynomial $F \in \FF_q[X_0, X_1; Y_0, Y_1]_{d,e}$ and $C = V(F)$. The degree of $C$ is given by the bidegree $(d, e)$ of $F$.

\subsection{Space filling curve}\label{ssec:spacefillingcurve}

In his influential paper \cite{Poo04}, Poonen proved a version of Bertini's theorem over finite fields. As a corollary, he obtained the following result.

\begin{theorem}{\protect{\cite[Corollary 3.5]{Poo04}}}
Let $X$ be a smooth, projective, geometrically integral variety over $\FF_q$. Then there exists a smooth, projective, geometrically integral curve $C \subset X$ such that $C(\FF_q) = X(\FF_q)$. 
\end{theorem}

In other words, over a finite field, we can find a smooth algebraic space filling curve, no matter how large the ambient space $X$ is. For later purposes, we leave a formal definition. 

\begin{definition}\label{def:spacefillingcurve}
Let $X \subset \PP^n$ be a smooth projective variety over $\FF_q$. A subscheme $C \subset X$ is called a \textbf{space filling curve} if $\dim C = 1$ and $C(\FF_q) = X(\FF_q)$.
\end{definition}

A scheme $X$ is \textbf{geometrically integral} if its extension $X_{\overline{\FF}_q}$ is integral. A $d$-dimensional scheme $X \subset \PP^n$ is \textbf{smooth} if for every (not-necessarily $\FF_q$-rational) point $P \in X(\overline{\FF}_q)$, the Jacobian matrix obtained from local equations is of rank $n-d$. 

\begin{example}\label{ex:P1P1}
Let $C \subset \PP^1 \times \PP^1 \subset \PP^3$ be a curve defined by a single bihomogeneous polynomial $F \in \FF_q[X_0, X_1; Y_0, Y_1]_{d,e}$. Then $C$ is smooth if and only if there is no point $P \in C(\overline{\FF}_q)$ such that 
\[
    F(P) = F_{X_0}(P) = F_{X_1}(P) = F_{Y_0}(P) = F_{Y_1}(P) = 0.
\]
\end{example}

Poonen's proof is probabilistic. He shows that if we take a sufficiently large degree complete intersection of hypersurfaces passing through all $\FF_q$-rational points on $X$, the probability of having a smooth and irreducible intersection is positive. Thus, the proof does not provide any concrete example. 

In literature, there have been only a few constructions. For $\PP^2$, \cite{HK13} classified all smooth space filling curves with minimal degree. Some higher-degree smooth space filling curves are studied in \cite{AG23}. For $\PP^3$, in \cite{CD+23}, some concrete examples for small primes were constructed. In the next section, we summarize the results of Homma and Kim on $\PP^1 \times \PP^1$.

\subsection{Space filling curves on $\PP^1 \times \PP^1$}\label{ssec:HommaKim}

In this section, we give a brief review of the work of Homma and Kim in \cite{HK23} on the existence of a smooth space filling curve on $(\PP^1 \times \PP^1)_{\FF_q}$.

Recall that a curve $C \subset (\PP^1 \times \PP^1)_{\FF_q}$ is determined by a bihomogeneous polynomial $F \in \FF_q[X_0, X_1; Y_0, Y_1]_{d,e}$. Homma and Kim showed that $C$ is a space filling curve if and only if $F$ is of the form 
\begin{equation}\label{eqn:spcstandard}
    F = f\cdot (X_0^q X_1 - X_0 X_1^q) + g \cdot (Y_0^q Y_1 - Y_0 Y_1^q),
\end{equation}
for two bihomogeneous polynomials $f, g \in \FF_q[X_0, X_1; Y_0, Y_1]$. As a result, the minimal degree of a space filling curve in $(\PP^1 \times \PP^1)_{\FF_q}$ is $(q+1, q+1)$, when $f \in \FF_q[Y_0, Y_1]_{q+1}$ and $g \in \FF_q[X_0, X_1]_{q+1}$.

\begin{definition}\label{def:Cfg}
For $f \in \FF_q[Y_0, Y_1]_{q+1}$ and $g \in \FF_q[X_0, X_1]_{q+1}$, let $C_{f,g}$ be the space filling curve of $\PP^1 \times \PP^1$ defined by the bihomogeneous polynomial $F$ in \eqref{eqn:spcstandard}.
\end{definition}

The next question is whether, among these minimal degree space filling curves, one may find a smooth one. The answer is affirmative. 

\begin{theorem}\protect{\cite[Theorem 3.3]{HK23}}\label{thm:HK23}
For any $q \ge 3$, there is a smooth space filling curve $C \subset (\PP^1 \times \PP^1)_{\FF_q}$ of bidegree $(q+1,q+1)$.
\end{theorem}

\begin{remark}\label{rmk:irreducible}
Note that once we have a smooth space filling curve $C_{f,g}$, then we don't need to check its irreducibility. If $C_{f,g}$ is reducible, then as its bidegree $(d, e)$ satisfies $d > 0$ and $e > 0$, its irreducible components have nontrivial intersection. Thus, on the intersection point, $C_{f,g}$ must be singular. 
\end{remark}

\begin{remark}\label{rmk:char2}
When $q = 2$, there is no smooth space filling curve of bidegree $(3, 3)$ \cite[Theorem 4.1]{HK23}. After embedding $\PP^1 \times \PP^1$ into $\PP^3$ by the Segre map, we may understand a bidegree $(d, e)$ curve $C$ as a degree $d+e$ curve in $\PP^3$. Homma showed that a nondegenerate irreducible degree $k$ curve in $\PP^3$ satisfies 
\[
    |C(\FF_q)| \le \frac{(q-1)(q^4-1)}{q(q^3-1)-3(q-1)}k
\]
\cite[Theorem 3.2]{Hom12}. When $q = 2$ and $k = 6$, the right hand side is $\frac{90}{11}$, that is strictly smaller than $|(\PP^1 \times \PP^1)(\FF_2)| = 9$. Therefore, it is not possible to have a space filling curve of bidegree $(3,3)$. The minimal degree of a smooth space filling curve over $\FF_2$ is $(4,3)$ \cite[Section 4.1]{HK23}. 
\end{remark}

\section{Smoothness criterion}\label{sec:smoothness}

As we have seen in Section \ref{ssec:HommaKim}, for any $f \in \FF_q[Y_0, Y_1]_{q+1}$ and $g \in \FF_q[X_0, X_1]_{q+1}$, $C_{f,g} = V(F)$ (see \eqref{eqn:spcstandard} for the definition of $F$) is a space filling curve.  It is necessary to have a smoothness criterion. Homma and Kim gave such a criterion in \cite[Lemma 3.2]{HK23}. In this section, we provide a slightly extended version of the criterion, that could also be applied to polynomials with points with higher multiplicity. For the reader's convenience, here we leave the proof. 

\begin{lemma}\label{lem:smoothness}
Let $f \in \FF_q[Y_0, Y_1]_{q+1}$ and $g \in \FF_q[X_0, X_1]_{q+1}$. Suppose that both $V(f)$ and $V(g)$ have no $\FF_q$-rational points. Then a point $P$ on $C_{f,g}$ is singular if and only if it satisfies 
\begin{equation}\label{eqn:rankonecondition}
\frac{g_{X_1}}{X_0^q} = -\frac{g_{X_0}}{X_1^q}
= -\frac{f_{Y_1}}{Y_0^q} = \frac{f_{Y_0}}{Y_1^q}.
\end{equation}
\end{lemma}

\begin{remark}
Another way to describe the condition \eqref{eqn:rankonecondition} is that the following $2 \times 4$ matrix is not of full-rank: 
\[
    \left[\begin{array}{rrrr}
    g_{X_1} & -g_{X_0} & -f_{Y_1} & f_{Y_0}\\
    X_0^q & X_1^q & Y_0^q & Y_1^q
    \end{array}\right].
\]
\end{remark}

We first prove another lemma. 

\begin{lemma}\label{lem:nonzerocoord}
Let $f \in \FF_q[Y_0, Y_1]_{q+1}$ and $g \in \FF_q[X_0, X_1]_{q+1}$. Assume that $V(f)$ and $V(g)$ have no $\FF_q$-rational points. If $P := (\alpha_0:\alpha_1) \times (\beta_0:\beta_1) \in C_{f,g}(\overline{\FF}_q)$ is a singular point, then both $(\alpha_0:\alpha_1)$, $(\beta_0:\beta_1)$ are not in $\PP^1(\FF_q)$. In particular, none of its coordinates are zero. 
\end{lemma}

\begin{proof}
Suppose that, on the contrary, $(\alpha_0:\alpha_1)\in\PP^1(\FF_q)$. Recall that for any $\FF_q$-rational point $(\alpha_0:\alpha_1)$, $\alpha_0^q \alpha_1 - \alpha_0 \alpha_1^q = \alpha_0 \alpha_1 - \alpha_0 \alpha_1 = 0$. Then
\[
    0 = F_{Y_0}(P)=f_{Y_0}(\beta_0, \beta_1)(\alpha_0^q \alpha_1 - \alpha_0 \alpha_1^q)  -\beta_1^q g(\alpha_0,\alpha_1)= -\beta_1^q g(\alpha_0,\alpha_1),
\]
and similarly, $\beta_0^qg(\alpha_0,\alpha_1)=F_{Y_1}(P) = 0$. Therefore, $g(\alpha_0,\alpha_1)=0$ and $V(g)$ has an $\FF_q$-rational point. So we have a contradiction. By symmetry, we also obtain $(\beta_0:\beta_1) \notin \PP^1(\FF_q)$.
\end{proof}

\begin{proof}[Proof of Lemma \ref{lem:smoothness}]
Suppose that $P:=(\alpha_{0} : \alpha_{1}) \times (\beta_{0}:\beta_{1}) \in (\PP^{1} \times \PP^{1})(\overline{\FF}_q)$ is a singular point of $C_{f,g} = V(F)$. Hence, $F(P)=F_{X_0}(P)=F_{X_1}(P)=F_{Y_0}(P)=F_{Y_1}(P)=0$. By Lemma \ref{lem:nonzerocoord}, we know that $(\alpha_0:\alpha_1), (\beta_0:\beta_1) \notin \PP^1(\FF_q)$. We split it into two cases. 

\textbf{Case 1.} Suppose that $f(\beta_0, \beta_1) = 0$.

In this case, we obtain 
\[
\begin{split}
    0 &= F(P)=g(\alpha_0,\alpha_1)(\beta_0^q\beta_1-\beta_0\beta_1^q),\\
    0 &= F_{X_0}(P)=g_{X_0}(\alpha_0,\alpha_1)(\beta_0^q\beta_1-\beta_0\beta_1^q),\\
    0 &= F_{X_1}(P)=g_{X_1}(\alpha_0,\alpha_1)(\beta_0^q\beta_1-\beta_0\beta_1^q).
\end{split}
\]
Since $(\beta_0:\beta_1)\notin\PP^1(\FF_q)$, we have $\beta_0^q\beta_1 - \beta_0 \beta_1^q \ne 0$. Thus, $g(\alpha_0, \alpha_1) = g_{X_0}(\alpha_0, \alpha_1) = g_{X_1}(\alpha_0, \alpha_1) = 0$. Because $g(\alpha_0, \alpha_1) = 0$, 
\[
\begin{split}
    0 &= F_{Y_0}(P)=f_{Y_0}(\beta_0,\beta_1)(\alpha_0^q\alpha_1-\alpha_0\alpha_1^q)=0,\\
    0 & = F_{Y_1}(P)=f_{Y_1}(\beta_0,\beta_1)(\alpha_0^q\alpha_1-\alpha_0\alpha_1^q)=0.
\end{split}
\]    
Since $(\alpha_0:\alpha_1)\notin\PP^1(\FF_q)$, we also have $f_{Y_0}(\beta_0,\beta_1) = f_{Y_1}(\beta_0,\beta_1) = 0$. Therefore, at $P$, \eqref{eqn:rankonecondition} holds because all numerators are zero. 

\textbf{Case 2.} Assume $f(\beta_0,\beta_1)\neq 0$. 

In this case, $g(\alpha_0,\alpha_1)\neq0$ too, by the same argument of Case 1 after swapping the roles of $f$ and $g$. 

We claim that $g_{X_0}(\alpha_0,\alpha_1)\neq0$, $g_{X_1}(\alpha_0,\alpha_1)\neq0$, $f_{Y_0}(\beta_0,\beta_1)\neq0$, and $f_{Y_1}(\beta_0,\beta_1)\neq0$. To see why, suppose that $g_{X_0}(\alpha_0,\alpha_1)=0$. Then $0 = F_{X_0}(P)=-\alpha_1^qf(\beta_0,\beta_1)$. Since $f(\beta_0,\beta_1)\neq0$, it follows that $\alpha_1=0$. Hence, $(\alpha_0:\alpha_1)=(1:0)$, a contradiction to the fact that $(\alpha_0:\alpha_1) \notin \PP^1(\FF_q)$. Similar reasoning follows for the other cases. 

From $F_{X_0}(P)=F_{X_1}(P)=F_{Y_0}(P)=F_{Y_1}(P)=0$, we have
\[
    \beta_0^q\beta_1-\beta_0\beta_1^q=\frac{\alpha_1^qf(\beta_0,\beta_1)}{g_{X_0}(\alpha_0,\alpha_1)}=-\frac{\alpha_0^qf(\beta_0,\beta_1)}{g_{X_1}(\alpha_0,\alpha_1)},
\]
\[
    \alpha_0^q\alpha_1-\alpha_0\alpha_1^q=\frac{\beta_1^qg(\alpha_0,\alpha_1)}{f_{Y_0}(\beta_0,\beta_1)}=-\frac{\beta_0^qg(\alpha_0,\alpha_1)}{f_{Y_1}(\beta_0,\beta_1)}.
\]
Since $f(\beta_0,\beta_1)\neq0$ and $g(\alpha_0,\alpha_1)\neq0$, the point $P$ satisfies 
\[
    \frac{X_1^q}{g_{X_0}}=-\frac{X_0^q}{g_{X_1}}, \qquad \frac{Y_1^q}{f_{Y_0}}=-\frac{Y_0^q}{f_{Y_1}},
\]
    
or equivalently, the first and the third equalities in \eqref{eqn:rankonecondition}.

Finally, substituting $\displaystyle \beta_0^q \beta_1 - \beta_0 \beta_1^q = \frac{\alpha_1^qf(\beta_0,\beta_1)}{g_{X_0}(\alpha_0,\alpha_1)}$ and $\displaystyle \alpha_0^q\alpha_1 - \alpha_0 \alpha_1^q = \frac{\beta_1^qg(\alpha_0,\alpha_1)}{f_{Y_0}(\beta_0,\beta_1)}$ into $F(P) = 0$ yields
\[
    \frac{f(\beta_0,\beta_1)\beta_1^qg(\alpha_0,\alpha_1)}{f_{Y_0}(\beta_0,\beta_1)}=-\frac{g(\alpha_0,\alpha_1)\alpha_1^qf(\beta_0,\beta_1)}{g_{X_0}(\alpha_0,\alpha_1)}.
\]
Since $f(\beta_0, \beta_1)g(\alpha_0, \alpha_1) \ne 0$, the point $P$ also satisfies $\displaystyle \frac{X_1^q}{g_{X_0}}=-\frac{Y_1^q}{f_{Y_0}}$. Therefore, $P$ satisfies  \eqref{eqn:rankonecondition}. 

Conversely, suppose that a point $P=(\alpha_{0} : \alpha_{1}) \times (\beta_{0}:\beta_{1}) \in (\PP^{1} \times \PP^{1})(\overline{\FF}_q)$ satisfies \eqref{eqn:rankonecondition}. The equations in \eqref{eqn:rankonecondition} imply $(X_1Y_1)(X_0^qf_{Y_1}+Y_0^qg_{X_1})=0$, $(X_1Y_0)(X_0^qf_{Y_0}+Y_1^qg_{X_1})=0$, $(-X_0Y_1)(X_1^qf_{Y_1}-Y_0^qg_{X_0})=0$, and $(-X_0Y_0)(X_1^qf_{Y_0}-Y_1^qg_{X_0})=0$ has a common solution $P$. Taking the sum of these four expressions, we have 
\[
    (Y_0f_{Y_0}+Y_1f_{Y_1})(X_0^qX_1-X_0X_1^q)+(X_0g_{X_0}+X_1g_{X_1})(Y_0^qY_1-Y_0Y_1^q)=0
\]
has a common solution $P$. Therefore, $F(P)=0$ by Euler's identity. Similarly, the equations in \eqref{eqn:rankonecondition} also imply that 
\[
    (-Y_1)(X_1^qf_{Y_1}-Y_0^qg_{X_0})=0, \quad (-Y_0)(X_1^qf_{Y_0}-Y_1^qg_{X_0})=0
\]
has a common solution $P$. Taking the sum of these two expressions we have
\[
    -X_1^q(Y_0f_{Y_0}+Y_1f_{Y_1})+g_{X_0}(Y_0^qY_1-Y_0Y_1^q)=0
\]
has $P$ as a solution. Therefore, $F_{X_0}(P)=0$. Similar reasoning shows $F_{X_1}(P)=0$, $F_{Y_0}(P)=0$, and $F_{Y_1}(P)=0$. Thus, $P$ is a singular point on $C_{f,g}$.
\end{proof}


\section{Examples on small primes}\label{sec:smallprime}

When the order of the field $\FF_q$ is small, one may do an exhaustive search for the smooth space filling curve $C_{f,g}$. In this section, we exhibit the numerical results for $q = 3$ and $5$. 

Recall that there is a natural $\SL_2(\FF_q)$-action on $\PP^1_{\FF_q}$ and on $\FF_q[X_0, X_1]_d$. It induces an $(\SL_2\times \SL_2)(\FF_q)$-action on $(\PP^1 \times \PP^1)_{\FF_q}$ and on $\FF_q[X_0, X_1; Y_0, Y_1]_{d,e} \cong \FF_q[X_0, X_1]_d\otimes_{\FF_q} \FF_q[Y_0, Y_1]_e$. 

\begin{lemma}\label{lem:SL2action}
Let $C_{f,g}$ be a minimal degree space filling curve. Then $C_{f,g}$ is smooth if and only if $C_{Af, Bg}$ is smooth for every $(A, B) \in (\SL_2 \times \SL_2)(\FF_q)$.
\end{lemma}

\begin{proof}
Since the group $(\SL_2 \times \SL_2)(\FF_q)$ acts on $(\PP^1 \times \PP^1)_{\FF_q}$ as a change of projective coordinates, $C_{f,g}$ is smooth if and only if $(A, B)C_{f,g}$ is smooth. Now it is sufficient to check that $(A, B)C_{f,g} = C_{Af, Bg}$. 

The defining equation of $C_{f,g}$ is 
\[
    F = f(Y_0, Y_1)(X_0^q X_1 - X_0 X_1^q) + g(X_0, X_1)(Y_0^q Y_1 - Y_0 Y_1^q) \in \FF_q[X_0, X_1;Y_0, Y_1]_{q+1,q+1}.
\]
The action of $(A, B) \in (\SL_2 \times \SL_2)(\FF_q)$ is given by 
\[
    (A, B)F = Af(Y_0, Y_1)B(X_0^q X_1 - X_0 X_1^q) + Ag(X_0, X_1)B(Y_0^q Y_1 - Y_0 Y_1^q).
\]
But for any
\[
    A = \left[\begin{array}{rr}a&b\\c&d\end{array}\right]\in \SL_2(\FF_q),
\]
\[
\begin{split}
    A (X_0^q X_1 - X_0 X_1^q) &= (aX_0 + bX_1)^q (cX_0 + dX_1) - (aX_0 + bX_1)(cX_0 + dX_1)^q\\
    &= (a^qX_0^q + b^q X_1^q)(cX_0 + dX_1) - (aX_0 + bX_1)(c^qX_0^q + d^qX_1^q)\\
    & = (a X_0^q + b X_1^q)(cX_0 + dX_1) - (aX_0 + bX_1)(c X_0^q + d X_1^q)\\
    &= (ad - bc)(X_0^q X_1 - X_0 X_1^q) = (X_0^q X_1 - X_0 X_1^q).
\end{split}
\]
Similarly, $B(Y_0^q Y_1 - Y_0 Y_1^q) = (Y_0^q Y_1 - Y_0 Y_1^q)$. Thus, 
\[
    (A, B)F = Af(Y_0, Y_1)(X_0^q X_1 - X_0 X_1^q) + Ag(X_0, X_1)(Y_0^q Y_1 - Y_0 Y_1^q)
\]
and $(A, B)C_{f,g} = C_{Af, Bg}$.
\end{proof}

\begin{remark}
We may not use the $(\PGL_2 \times \PGL_2)(\FF_q)$-action for a similar statement. This is because $\PGL_2(\FF_q)$ does not act naturally on $\FF_q[X_0, X_1]_{q+1}$. We also may not use $(\GL_2 \times \GL_2)(\FF_q)$, because $\GL_2(\FF_q)$ acts on $X_0^q X_1 - X_0 X_1^q$ nontrivially.
\end{remark}

Lemma \ref{lem:SL2action} allows us to reduce the computational complexity significantly. Instead of working over all possible pairs of polynomials $f \in \FF_q[Y_0, Y_1]_{q+1}$ and $g \in \FF_q[X_0, X_1]_{q+1}$, we may work with a pair chosen from two polynomials separately chosen from $\SL_2(\FF_q)$-orbits in $\FF_q[X_0, X_1]_{q+1}$. 

\subsection{$q = 3$}\label{ssec:q=3}

Let $\cG_q \subset \FF_q[X_0, X_1]_{q+1}$ be the subset consisting of polynomials without any $\FF_q$-rational points. Then $\cG_q$ admits an induced $\SL_2(\FF_q)$-action. 

For $q = 3$, with computer-aided computation, we obtain the following results. There are seven $\SL_2(\FF_3)$-orbits on $\cG_3$, summarized below. 
\begin{enumerate}
\item $O_{2^2}$ is the orbit of $X_0^4 - X_0^2 X_1^2 + X_1^4$. They are squares of a monic irreducible quadratic polynomial. $|O_{2^2}| = 3$.
\item $-O_{2^2}$ is the orbit of $-X_0^4 + X_0^2 X_1^2 - X_1^4$, the negative of a square of a monic quadratic irreducible polynomial. $|-O_{2^2}| = 3$. 
\item $O_{2+2}$ is the orbit of $X_0^4 + X_1^4$, consisting of quartic polynomials that are products of two distinct irreducible quadratic polynomials. $|O_{2+2}| = 6$.
\item $O_{4a}$ is the orbit of $X_0^4 + X_0 X_1^3 - X_1^4$, which is an irreducible quartic polynomial. $|O_{4a}| = 6$.
\item $O_{4b}$ is the orbit of $X_0^4 - X_0 X_1^3 - X_1^4$, which is also an irreducible polynomial.$|O_{4b}| = 6$.
\item $O_{4c}$ is the orbit of $X_0^4 + X_0^2 X_1^2 - X_1^4$, which is irreducible. $|O_{4c}| = 12$. 
\item $O_{4d}$ is the orbit of $X_0^4 - X_0^2 X_1^2 - X_1^4$, which is also irreducible. $|O_{4d}| = 12$.
\end{enumerate}

For each $g \in \cG_q$, we may consider the set 
\[
    \cF(g) := \{f \in \FF_q[Y_0, Y_1]_{q+1}\;|\; C_{f,g} \mbox{ is a smooth space filling curve}\}. 
\]
Lemma \ref{lem:SL2action} implies that if $g_1$ and $g_2$ are in the same $\SL_2(\FF_q)$-orbit, then $\cF(g_1) = \cF(g_2)$. So for an $\SL_2(\FF_q)$-orbit $O$, the notation $\cF(O)$ is well-defined. Moreover, by Lemma \ref{lem:SL2action} again, each $\cF(O)$ admits an $\SL_2(\FF_q)$-action, hence it is decomposed into $\SL_2(\FF_q)$-orbits. We obtain the following decomposition.
\[
\begin{split}
    \cF(O_{2^2}) = \cF(-O_{2^2}) &= O_{2+2} \sqcup O_{4a} \sqcup O_{4b} \sqcup O_{4c} \sqcup O_{4d},\\
    \cF(O_{2+2}) &= O_{2^2} \sqcup -O_{2^2} \sqcup O_{4a} \sqcup O_{4b} \sqcup O_{4c} \sqcup O_{4d}, \\
    \cF(O_{4a}) &= O_{2^2} \sqcup -O_{2^2} \sqcup \cO_{2+2} \sqcup O_{4a} \sqcup O_{4c} \sqcup O_{4d},\\
    \cF(O_{4b}) &= O_{2^2} \sqcup -O_{2^2} \sqcup O_{2+2} \sqcup O_{4b} \sqcup O_{4c} \sqcup O_{4d}, \\
    \cF(O_{4c}) = \cF(O_{4d}) &= O_{2^2} \sqcup -O_{2^2} \sqcup O_{2+2} \sqcup O_{4a} \sqcup O_{4b}. 
\end{split}
\]

In summary, there are 1,584 smooth space filling curves for $q = 3$, while the set of pairs of polynomials $(f, g)$ without any $\FF_q$-rational points has $48^2 = 2,\!304$ elements. 

\begin{remark}\label{rem:observations}
We may have a few observations. 
\begin{enumerate}
\item For any $g \in \cG_3$, $\cF(g)$ is nonempty. In other words, for any $g$ without any $\FF_3$-rational point, there is $f$ such that $C_{f,g}$ is a smooth space filling curve. 
\item Secondly, the size of the set $\cF(O)$ for each orbit is large -- for each case, only one or two orbits in $\cG_3$ are excluded. 
\item Finally, observe that $O_{4a} \subset \cF(O_{4a})$. This implies that there is a smooth space filling curve of the form $C_{f,f}$, which will be discussed in Section \ref{sec:symmetric} more carefully. 
\end{enumerate}
\end{remark}

\subsection{$q=5$}\label{ssec:q=5}

Using Lemma \ref{lem:SL2action}, we may significantly reduce the size of computation, and can obtain an exhaustive computational result for $q = 5$. First of all,
$\cG_5$, the set of degree 6 homogeneous polynomials without any $\FF_5$-rational points, has $20,\!480$ elements. We can compute its $\SL_2(\FF_5)$-orbit decomposition $\cG_5 = \bigsqcup O_i$. For each orbit $O_i$, we choose a polynomial $f_i \in O_i$, and tested if $C_{f_i, f_j}$ is a smooth space filling curve or not. As a result, we observe that for any orbit $O_i$, $\cF(O_i)$ is nonempty. Furthermore, the total number of smooth space filling curves over $\FF_5$ is $412,\!004,\!800$. Note that the total number of pairs $(f, g)$ without $\FF_5$-rational points is $20,\!480^2 = 419,\!430,\!400$. Therefore, more than $98\%$ of pairs provide smooth space filling curves. You may find a list of pairs $(f,g)$ such that $C_{f,g}$ is singular on \cite{ACM25}. 

\subsection{Statistics in small prime powers}\label{ssec:statistics}

From the numerical data above, we may expect that minimal degree smooth space filling curves are abundant. Using Macaulay2, we ran a numerical experiment. For some small prime powers, we constructed 1,000 random pairs of polynomials $f \in \FF_q[Y_0, Y_1]_{q+1}$ and $g \in \FF_q[X_0, X_1]_{q+1}$ without any $\FF_q$-rational points on $V(f)$ and $V(g)$, and checked if $C_{f,g}$ is smooth. The source code is available on \cite{ACM25}. Table \ref{tbl:smoothpairs} shows that the probability to have a smooth curve converges to 1 very quickly.

\begin{table}[!ht]
\begin{tabular}{|c|r||c|r|} 
 \hline
 $q$  &  Smooth 
 & 
 $q$  &  Smooth 
 \\ [0.5ex] 
 \hline\hline
 3 & 676 & 16 & 998 \\
 \hline
 4 & 935 & 17 & 1,000 \\  
 \hline
 5 & 981 & 19 & 1,000 \\ 
 \hline
 7 & 994 &  23 & 1,000 \\ 
 \hline
 8 & 992 &   25 & 999\\
 \hline
 9 & 995 & 27 & 1,000\\
 \hline
 11 & 995 & 29 & 999 \\ 
 \hline
 13 & 996 & 31 & 1,000 \\
 \hline
\end{tabular}
\medskip
\caption{The number of smooth space filling curves from 1,000 random samples.}
\label{tbl:smoothpairs}
\end{table}


\section{Proof of the main theorem}\label{sec:construction}

In this section, we prove Theorem \ref{thm:mainthm}. Let $f \in \FF_q[Y_0, Y_1]_{q+1}$ be a homogeneous polynomial such that $V(f)$ has no $\FF_q$-rational points. We will construct an explicit $g \in \FF_q[X_0, X_1]_{g+1}$ such that $C_{f,g}$ is smooth. 

\begin{proof}[Proof of Theorem \ref{thm:mainthm}]
We already have the result for $q \le 5$ from Section \ref{sec:smallprime}. We assume $q > 5$. 

Let $g(X_0, X_1) = X_0^{q+1} + \lambda_1 X_0 X_1^q + \lambda_2 X_1^{q+1}$ where $\lambda_1,\lambda_2\in\FF_q^*$ and $\lambda_1^2-4\lambda_2$ is a non-square. This implies $V(g)$ does not have any $\FF_q$-rational points. Indeed, since $g(1, 0) = 1$, $(1:0) \notin V(g)(\FF_q)$. If $(\alpha:1) \in V(g)(\FF_q)$, 
\[
	0 = g(\alpha, 1) = \alpha^{q+1} + \lambda_{1} \alpha + \lambda_{2} = \alpha^{2} + \lambda_{1}\alpha + \lambda_{2}
\]
since $\alpha^{q} = \alpha$ for $\alpha \in \FF_{q}$. Thus, by the quadratic formula, $\lambda_{1}^{2} - 4\lambda_{2}$ is a square in $\FF_{q}$. For any $k \in \FF_q^*$, $V(kg)(\FF_q) = V(g)(\FF_q) = \emptyset$. Therefore, by Lemma \ref{lem:smoothness}, $C_{f, kg}$ is a smooth space filling curve if and only if \eqref{eqn:rankonecondition} has no $\overline{\FF}_{q}$-solution on $\PP^{1} \times \PP^{1}$ for the pair of polynomials $(f, kg)$. Note also that $\lambda_{2} \ne 0$, as otherwise $\lambda_{1}^{2} - 4\lambda_{2}= \lambda_{1}^{2}$ is clearly a square. There are many pairs $(\lambda_1, \lambda_2) \in (\FF_q^*)^2$ such that $\lambda_1^2 - 4\lambda_2$ is not a square. We will choose one of them explicitly so that $C_{f, kg}$ is smooth for some $k \in \FF_q^*$.

Suppose that $(\alpha_{0} : \alpha_{1}) \times (\beta_{0}:\beta_{1}) \in (\PP^{1} \times \PP^{1})(\overline{\FF}_q)$ is a solution of \eqref{eqn:rankonecondition}. Lemma \ref{lem:nonzerocoord} tells us none of the coordinates are zero. The first equality in \eqref{eqn:rankonecondition} is
\[
    \frac{k\lambda_2 \alpha_1^q}{\alpha_0^q} 
    = -\frac{k(\alpha_0^q + \lambda_1\alpha_1^q)}{\alpha_1^q}.
\]
Then we have
\[
    k(\alpha_0^2 + \lambda_1\alpha_0 \alpha_1 + \lambda_2\alpha_1^2)^q = k(\alpha_0^{2q} + \lambda_1\alpha_0^q \alpha_1^q + \lambda_2 \alpha_1^{2q}) = 0.
\]
From $k \ne 0$, setting $\phi := \alpha_1/\alpha_0$, we obtain 
\[
    \lambda_2 \phi^2 + \lambda_1\phi + 1 = 0.
\]
Hence, there are at most two solutions ($\phi \text{ and its Galois conjugate }\phi^q$), and they are in a degree two extension field $\FF_{q^{2}}$ of $\FF_{q}$. Since $\lambda_1^2 - 4\lambda_2$ is not a square by the assumption, we have $\phi \ne \phi^q$, and $\phi, \phi^q \in \FF_{q^2}\setminus \FF_q$.

On the other hand, $(\beta_0: \beta_1)$ is a solution of 
\[
    -\frac{f_{Y_1}}{Y_0^q} = \frac{f_{Y_0}}{Y_1^q},
\]
or equivalently, it satisfies 
\[
    Y_0^q f_{Y_0} + Y_1^q f_{Y_1}=0.
\]
Because it is a homogeneous polynomial of degree $2q$, $V(Y_0^q f_{Y_0} + Y_1^q f_{Y_1})(\overline{\FF}_q)$ has $2q$ points, counted with multiplicities. Set $\gamma_{j} = \beta_{1}/\beta_{0}$ for these $2q$ zeros where $1\le j\le 2q$. Since $(\alpha_0: \alpha_1) \times (\beta_0: \beta_1) \in (\PP^1 \times \PP^1)(\overline{\FF}_q)$ is a solution of \eqref{eqn:rankonecondition}, it must also be a solution of 
\begin{equation}\label{eqn:thirdequation}
\frac{k\lambda_2X_1^q}{X_0^q}=-\frac{f_{Y_1}}{Y_0^q}, 
\end{equation}
hence satisfies 
\[
	k\lambda_{2}\phi^{q} = \frac{k \lambda_{2}\alpha_{1}^{q}}{\alpha_{0}^{q}} = -\frac{f_{Y_{1}}(\beta_{0}, \beta_{1})}{\beta_{0}^{q}} = -f_{Y_{1}}(1, \gamma_{j}).
\]
Note that from our construction, $k, \lambda_2 \in \FF_q$, and $\phi^q \in \FF_{q^2}\setminus \FF_q$, hence $k\lambda_{2}\phi
^q\in \FF_{q^{2}} \setminus \FF_q$. The set  $\FF_{q^{2}}\setminus \FF_{q}$ is divided into $(q^{2}-q)/2$-Galois orbits, for the Galois group $G = G(\FF_{q^{2}}/\FF_{q}) \cong \ZZ/2\ZZ$. (Indeed, it is generated by the Frobenius map $F : a \mapsto a^q$). If $q > 5$, $(q^{2}-q)/2 > 2q$. Therefore, we can always find a Galois orbit $\{\phi_{1}, \phi_{2}\}$ that is disjoint from the set $\{-f_{Y_{1}}(1, \gamma_{j})\} \cap (\FF_{q^2} \setminus \FF_q)$, whose cardinality is at most $2q$. Take the irreducible monic polynomial $X^{2} + aX + b$ of $\phi_{1}, \phi_{2}$. Then its scalar multiple $b^{-1}X^{2} + ab^{-1}X + 1$ has the same zeros $\phi_{1}, \phi_{2}$. Now set $\lambda_{2} := b^{-1}$, $\lambda_{1} := ab^{-1}$, and $k := \lambda_{2}^{-1}$. From the irreducibility of $b^{-1}X^{2} + ab^{-1}X + 1 = \lambda_{2}X^{2} + \lambda_{1}X + 1$, we obtain $\lambda_{1}^{2} - 4\lambda_{2}$ is not a square. And $\{k\lambda_{2}\phi_{i}\} = \{\phi_{i}\}$ and this set is disjoint from $\{-f_{Y_{1}}(1, \gamma_{j})\}$ by the construction. Therefore, \eqref{eqn:thirdequation} does not have any solution, and therefore, $C_{f, kg}$ is smooth. 
\end{proof}

\begin{remark}
The proof does not work over characteristic two fields, because $\lambda_1^2 - 4\lambda_2 = \lambda_1^2$ is always a square. However, as numerical results in Section \ref{sec:smallprime} suggest, we expect that the statement holds for any $q > 2$.
\end{remark}

\section{Symmetric example}\label{sec:symmetric}

Theorem \ref{thm:mainthm} and numerical results in Section \ref{sec:smallprime} suggest that there are plenty of minimal degree smooth space filling curves. In particular, we may construct examples with extra symmetry. In this section, we prove Theorem \ref{thm:symmetric}.

\begin{definition}\label{def:symmetric}
For a pair of polynomials $f(Y_0, Y_1) \in \FF_q[Y_0, Y_1]_{q+1}$ and $g(X_0, X_1) \in \FF_q[X_0, X_1]_{q+1}$, 
a space filling curve $C_{f,g}$ is called \textbf{symmetric} if $g(X_0, X_1) = f(X_0, X_1)$, hence $C_{f,g} = C_{f,f}$.
\end{definition}

\begin{proof}[Proof of Theorem \ref{thm:symmetric}]
Let 
\[
    f(Y_0,Y_1)=Y_0^{q+1}+Y_0 Y_1^q+\lambda Y_1^{q+1}
\]
where $\lambda \in \FF_q\backslash\{-(u^2+u)|u\in \FF_q\}$, and set $g(X_0, X_1) = f(X_0, X_1)$. Since $t : \FF_q \to \FF_q$, $t(x) = -(x^2+x)$ is a degree two map ramified at a point, its image has $\frac{q-1}{2} + 1 = \frac{q+1}{2}$ points. Thus, there are $\frac{q-1}{2}$ choices of $\lambda$. 

We claim that $V(f)$ does not have any $\FF_q$-rational points. Note that $f(1,0)= 1 \neq0$. For the other $\FF_q$-points, we may set its coordinate to $(y:1)$. Then 
\[
    f(y,1)=y^{q+1}+y+\lambda =y^2+y+\lambda.
\]
Hence, $(y:1) \in V(f)$ only if $\lambda=-(y^2+y)$. From our choice of $\lambda$, this is not the case. 

We may use Lemma \ref{lem:smoothness} to show that $C_{f,f}$ is smooth. To do so, we need to show that 
\begin{equation}\label{eqn:symmetricsmooth}
    \frac{\lambda X_1^q}{X_0^q}=\frac{-X_0^q-X_1^q}{X_1^q}=\frac{-\lambda Y_1^q}{Y_0^q}=\frac{Y_0^q+Y_1^q}{Y_1^q}
\end{equation}
has no solutions in $(\PP^1\times\PP^1)(\overline{\FF}_q)$. 

Now, suppose that $P=(\alpha_0:\alpha_1)\times(\beta_0:\beta_1)\in(\PP^1\times\PP^1)(\overline{\FF}_q)$ satisfies \eqref{eqn:symmetricsmooth}. By Lemma \ref{lem:nonzerocoord}, we may assume that $\alpha_0, \alpha_1, \beta_0, \beta_1$ are all nonzero. Let $x := \alpha_1/\alpha_0$ and $y := \beta_1/\beta_0$. Then \eqref{eqn:symmetricsmooth} becomes 
\[
    \lambda x^q=-\frac{1}{x^q}-1=-\lambda y^q=\frac{1}{y^q}+1.
\]
Hence $x^q=-y^q$, which implies that 
\[
    \frac{1}{y^q}-1=\frac{1}{y^q}+1,
\]
a contradiction over an odd characteristic field. Thus, no point in $(\PP^1\times\PP^1)(\overline{\FF}_q)$ satisfies \eqref{eqn:symmetricsmooth}, hence $C_{f,f}$ is smooth.
\end{proof}

\begin{remark}\label{rmk:moresymmetricexamples}
A similar computation shows that there are many examples of symmetric smooth space filling curves. For instance, the following polynomials also provide examples when $\lambda \in \FF_q \setminus \{-(u^2+u)\;|\; u \in \FF_q\}$. 
\[
    Y_0^{q+1} - Y_0 Y_1^q + \lambda Y_1^{q+1}, \quad 
    Y_0^{q+1} + Y_0^q Y_1 + \lambda Y_1^{q+1}, \quad
    Y_0^{q+1} - Y_0^q Y_1 + \lambda Y_1^{q+1}.
\]
\end{remark}

\begin{remark}\label{rmk:char2symmetric}
From the numerical search in Section \ref{sec:smallprime}, we found that there is no symmetric smooth space filling curve for $q = 4$.
\end{remark}

\bibliographystyle{alpha}

\end{document}